\documentclass[12pt]{amsart}
\usepackage{amssymb, amscd, amsmath, amsthm, tikz-cd}
\usepackage[utf8]{inputenc}
\begin{document}

\theoremstyle{plain}
\newtheorem{thm}{Theorem}[section]
\newtheorem{lem}[thm]{Lemma}
\newtheorem{definition}[thm]{Definition}
\newtheorem{prop}[thm]{Proposition}
\newtheorem{cor}[thm]{Corollary}
\newtheorem{rem}[thm]{Remark}
\newtheorem{ex}[thm]{Example}
\long\def\alert#1{\smallskip{\hskip\parindent\vrule%
\vbox{\advance\hsize-2\parindent\hrule\smallskip\parindent.4\parindent%
\narrower\noindent#1\smallskip\hrule}\vrule\hfill}\smallskip}
\def\ff{\frak}
\def\Spec{\mbox{\rm Spec}}
\def\type{\mbox{ type}}
\def\Hom{\mbox{ Hom}}
\def\rank{\mbox{ rank}}
\def\Ext{\mbox{ Ext}}
\def\Ker{\mbox{ Ker}}
\def\Max{\mbox{\rm Max}}
\def\Cont{\mbox{\rm Cont}}
\def\Con{\mbox{\rm Con}}
\def\End{\mbox{\rm End}}
\def\l{\langle\:}
\def\r{\:\rangle}
\def\Rad{\mbox{\rm Rad}}
\def\Zar{\mbox{\rm Zar}}
\def\Supp{\mbox{\rm Supp}}
\def\Rep{\mbox{\rm Rep}}
\def\cal{\mathcal}
\title[MV-frames]{MV-frames}
%\thanks{\today}
\author{Jean B. Nganou}
\address{Department of Mathematics and Statistics, University of Houston-Downtown, TX 77014} \email{nganouj@uhd.edu}
\begin{abstract} Complete MV-algebras are naturally equipped with frame structures. We call them MV-frames and investigate some of their main the properties as frames. We completely characterized algebraic MV-frames as well as regular MV-frames. In addition, we consider nuclei on MV-frames in general and on MV-frames of ideals of \L ukasiewicz rings. Finally, we used the Chang-Mundici functor to explore the frame structures of complete unital lattice-ordered groups.
%\vspace{.2cm}\\
\end{abstract}
\maketitle
\section{Introduction}
An MV-algebra can be defined \cite{CDM} as an Abelian monoid $(A,\oplus , 0)$ with an involution $\neg:A\to A$ (i.e., $\neg \neg x=x$ for all $x\in A$) satisfying the following axioms for all $x, y\in A$: $\neg 0\oplus x=\neg0$,
$\neg(\neg x\oplus y)\oplus y=\neg(\neg y\oplus x)\oplus x$. For any $x,y\in A$, if one writes $x\leq y$ when $\neg x\oplus y=\neg 0:=1$, then $\leq$ induces a partial order on $A$, which is in fact a lattice order where $x\vee y=\neg(\neg x\oplus y)\oplus y$ and $x\wedge y=\neg(\neg x\vee \neg y)$.\par
Complete MV-algebras are known to satisfy the following distributive laws \cite[Lem. 6.6.4]{CDM}: for every $x\in A$ and every $X\subseteq A$
$$x\wedge \bigvee X=\bigvee (x\wedge X)\hspace{1cm} x\vee \bigwedge X=\bigwedge (x\vee X) $$
 A \textit{frame} is a complete lattice $L$ in which  the frame law holds: $$a\wedge \bigvee S=\bigvee\{a\wedge s\mid s\in S\},$$ for all $a\in L$ and $S\subseteq L$. The theory of frames has been for the past century one of the most active area of lattice theory. Frames deal with a framework in which important algebraic and topological properties of ideals, filters or congruences in rings, lattices or other algebras are investigated (see \cite{BANASCH1.96, BBBANASCH3.2003, DB, Picado}).\par
As observed above, every complete MV-algebra is both a frame and a dual frame. For uniformity in terminologies, as the terminology Boolean frames is used complete Boolean algebra viewed frames, we shall use the terminology MV-frames for complete MV-algebras treated as frames. This work is intended to be an introductory treatment of MV-frames where we explore some of the basic frame notions in the MV-algebraic framework. In particular, we investigate which of the MV-frames are algebraic and discovered that they are precisely the Stone MV-algebras, that the coherent MV-frames are finite MV-algebras and the algebraic and regular MV-frames are powerset algebras. In addition, we consider nuclei on MV-frames and their nuclear which under certain assumptions are MV-frames of their own. One important class of MV-frames treated is that of MV-frames of ideals of \L ukasiewviez rings \cite{BD}. In this case, the standard radical of ideals is proved to carry many important properties. Given that MV-algebras are categorically equivalent to abelian lattice-ordered groups with distinguished units \cite{CM}, it is only natural that we peak at the other side of the equivalence and consider some of the frame concepts in the context of complete abelian $\ell$-groups. We use the equivalence to introduce some frame concepts on complete abelian $\ell$-groups and characterize algebraic $\ell u$-frames.\par
The paper is organized as follows:\\
In section 2, we introduce the notion of MV-frames and investigated some of their main attributes. In particular, we determine compact elements of completely distributive MV-frames and use it to characterize algebraic MV-frames and obtain that they are precisely the direct products of finite MV-chains (Theorem \ref{algebraicMV}). Furthermore, we obtain that coherent MV-frames are finite (Corollary \ref{coherent}), that among the algebraic MV-frames, the regular ones are exactly the powerset Boolean algebras (Corollary \ref{regular}), and that coherent frame morphisms are the complete homomorphisms preserving maximal compact elements (Proposition \ref{coherentmap}).\par
In section 3, we introduce nuclei on MV-frames and consider various type of nuclei. We obtain that the radical on the MV-frame of ideals of a \L ukasiewicz ring is a nucleus that is inductive (Proposition \ref{radical}). We also obtain that the nuclear of any inductive nucleus of MV-type is an algebraic MV-frame (Proposition \ref{nuclear-p}).\par
In section 4, we use the Chang-Mundici equivalence to consider the frame concepts on lattice-ordered abelian groups. We introduce the notions of $\ell u$-frames, of compact elements and algebraic $\ell u$-frames. We prove in particular that MV-frames correspond to $\ell u$-frames (Proposition \ref{mv=lu} and that the Chang-Mundici equivalence restricts to an equivalence between MV-frames and $\ell u$-frames (Theorem \ref{mvf=luf}). We also use the equivalence to derive a characterization of all algebraic $\ell u$-frames (Theorem \ref{algebraicluf}).\par

Recall that an element $a$ of a frame $L$ is called \textit{compact} if for every $S\subseteq L$ such that $a\leq \bigvee S$, there exists a finite subset $F$ of $S$ such that $a\leq \bigvee F$. The set of compact elements of $L$ is. denoted by $\mathfrak{k}(L)$. 
If the top element $1$ of $L$ is compact it is said that $L$ is compact. In addition, $L$ is said to have the \textit{finite intersection property} (always abbreviated FIP) if for any pair $a, b\in \mathfrak{k}(L)$, it follows that $a\wedge b\in \mathfrak{k}(L)$. The frame $L$ is \textit{algebraic} if every $a\in L$ is the join of compact elements below it ($a=\bigvee\{x\in \mathfrak{k}(L): x\leq a\}$). For every $x, y\in L$, we write $x\preceq y$ (read $x$ is way below $y$) if $y\vee x^\ast=1$ (where $x^\ast$ is the pseudo complement of $x$). An element $x\in L$ is called \textit{regular} if $x=\bigvee\{a\in L: a\preceq x\}$. The frame $L$ is \textit{regular} if all of its elements are regular; $L$ is \textit{coherent} if it is compact, algebraic and has the FIP. Given two frames $L_1$ and $L_2$, a \textit{frame homomorphism} from $L_1\to L_2$ is any lattice homomorphism that preserves arbitrary joins. A frame homomorphism $f:L_1\to L_2$ is called coherent if $f(\mathfrak{k}(L_1))\subseteq \mathfrak{k}(L_2)$. Details about frames, their basic terminologies and results can be found in the following references \cite{Dube, DB, Jo, MZ1, MZ}. Basic notions of MV-algebras can be reviewed in the classic texts \cite{CDM, Mu}.

Throughout the paper, $A$ will denote a complete MV-algebra and $B(A)$ the Boolean center of $A$, that is the largest Boolean subalgebra of $A$.
\section{Coherent MV-algebras}
In this section, we aim to determining the intersection of the class of MV-frames and known classes of frames such as algebraic and coherent. In other words, we wish to characterize all MV-frames that are algebraic and all MV-frames that are coherent. 

Recall that for $z\in A$, the pseudocomplement of $z$ is defined as $z^\ast=\bigvee\{x|x\wedge z=0\}$. This pseudocomplementation is known to carry the following properties (see \cite[p.133]{CDM}). (P.1) $z^\ast\in B(A)$, for all $z\in A$; (P.2) $z\leq z^{\ast\ast}$, for all $z\in A$; (P.3) $x\leq y$ implies $x^{\ast\ast}\leq y^{\ast\ast}$, for all $x, y\in A$; (P.4) $z^{\ast\ast}=z$, for all $z\in B(A)$.

We begin with the following lemma. 
\begin{lem}\label{compact}
For every $x\in A$, if $x\in \mathfrak{k}(A)$, then $x^{\ast\ast}\in \mathfrak{k}(B(A))$.
\end{lem}
\begin{proof}
Let $x\in \mathfrak{k}(A)$. Note that by P.1, $x^{\ast\ast}\in B(A)$. Moreover, suppose that $x^{\ast\ast}\leq \bigvee S$ for some $S\subseteq B(A)$. Then, by P.2 $x\leq \bigvee S$ and since $x$ is compact in $A$, there exists a finite subset $T$ of $S$ such that $x\leq \bigvee T$. Note that as $T\subseteq B(A)$, then $\bigvee T\in B(A)$. In addition, as $x\leq \bigvee T$, it follows from P.3 and P.4 that $x^{\ast\ast}\leq (\bigvee T)^{\ast\ast}=\bigvee T$. Therefore $x^{\ast\ast}$ is compact in $B(A)$ as required.
\end{proof}
\begin{prop}\label{BF}
If $A$ is an algebraic MV-frame, then $B(A)$ is an algebraic Boolean frame.
\end{prop}
\begin{proof}
First, as $A$ is complete, $B(A)$ is a complete Boolean algebra \cite[Cor. 6.6.5]{CDM}. Suppose that $A$ is algebraic and let $x\in B(A)$. Then $x=\bigvee\{a\in \mathfrak{k}(A): a\leq x\}$. Note that for every $a\in \mathfrak{k}(A)$ such that $a\leq x$, by Lemma \ref{compact}, P.3 and P.4,  $a^{\ast\ast}\in \mathfrak{k}(B(A))$ and  $a^{\ast\ast}\leq x$. It follows that $\bigvee\{a^{\ast\ast}: a\in \mathfrak{k}(A), a\leq x\}\leq \bigvee\{a\in \mathfrak{k}(B(A)): a\leq x\}$. But since $a\leq a^{\ast\ast}$, then $x=\bigvee\{a\in \mathfrak{k}(A): a\leq x\}\leq \bigvee\{a^{\ast\ast}: a\in \mathfrak{k}(A), a\leq x\}\leq \bigvee\{a\in \mathfrak{k}(B(A)): a\leq x\}$. Thus, $x\leq \bigvee\{a\in \mathfrak{k}(B(A)): a\leq x\}\leq x$ and $x=\bigvee\{a\in \mathfrak{k}(B(A)): a\leq x\}$. Whence, $B(A)$ is an algebraic Boolean frame.
\end{proof}
\begin{cor}\label{alg=ch}
Every algebraic MV-frame is isomorphic to a direct product of the form $\prod_{k\in K}A_k$, where each $A_k$ is either a finite MV-chain or the standard MV-algebra $[0,1]$. 
\end{cor} 
\begin{proof}
Let $A$ be an algebraic MV-frame. Then, $B(A)$ is an algebraic Boolean frame by Proposition \ref{BF}. Recall that in a complete Boolean algebra, the compact elements are precisely the finite joins of atoms. It follows that algebraic Boolean frames must be atomic. Thus, as $A$ is a complete MV-algebra, by the above and  \cite[Cor. 6.6.5]{CDM}, $B(A)$ is a complete and atomic Boolean algebra \cite[Cor. 6.6.5]{CDM}. It follows from \cite[Thm. 6.8.1]{CDM} that $A$ is complete and completely distributive. The conclusion is clear from \cite[Thm. 2.2]{jbn2}. 
\end{proof}
\begin{prop}\label{compact-chmv}
Let $A:=\prod_{k\in K}A_k$, where each $A_k$ is either a finite MV-chain or the standard MV-algebra $[0,1]$.\\ Let $I=\{k\in K: A_k \; \mbox{is a finite}\; MV-\mbox{chain}\}$ and $J=\{k\in K:A_k=[0,1]$.\\
Then, for every $\alpha\in A$, $\alpha$ is compact if and only if (c1) $\alpha\in \oplus_{k\in K}A_k$, and (c2) $\{k\in K: \alpha(k)\ne 0\}\subseteq I$.
\end{prop} 
\begin{proof}
$\Rightarrow):$ Suppose that $\alpha$ is compact in $A$. For each $k\in K$, define $\beta_k\in A$ by: 
$$
\beta_k(j)=
\begin{cases}
\alpha(k), \mbox{if}\; j=k\\
0,\; \; \; \; \; \; \mbox{if}\; j\ne k
\end{cases}
$$
Then $\alpha=\bigvee\{\beta_k: k\in K\}$ and since $\alpha$ is compact, there exists a finite subset $F$ of $K$ such that $\alpha\leq\bigvee\{\beta_k: k\in F\}$. It follows that $\alpha(k)=0$ for all $k\notin F$ since $\beta(k)=0$ for all $k\notin F$ and the suprema in $A$ are computed coordinate-wise. Thus, $\alpha\in \oplus_{k\in K}A_k$, which is $(c1)$. For $(c2)$, suppose by contradiction that there exists $j\in J$ such that $\alpha(j)\ne 0$. Then, choose a strictly increasing sequence $(x_{jn})_n\subseteq [0,1]$ that converges to $\alpha(j)$. Now, define the sequence $(\beta_n)_n\subseteq A$ by:
$$
\beta_n(k)=
\begin{cases}
x_{jn}, \mbox{if}\; k=j\\
\alpha(k),\; \mbox{if}\; k\ne j
\end{cases}
$$
Then, $\alpha\leq \bigvee\{\beta_n: n\in \mathbb{N}\}$ and there are no finite subsets $F\subseteq \mathbb{N}$ such that $\alpha\leq \bigvee\{\beta_n: n\in F\}$. This contradicts the compactness of $\alpha$ and $\alpha(k)=0$ for all $k\in J$.\\
$\Leftarrow):$ Suppose that $\alpha\in A$ satisfies $(c1), (c2)$. Let $F:=\{k\in K: \alpha(k)\ne 0\}$. Now, let $\alpha\leq \bigvee\{\beta_x: x\in X\}$ for some $X\subseteq K$. For each $k\in F$, $k\in I$ and $(\bigvee\{\beta_x: x\in X\})(k)=\mbox{Max}\{\beta_x(k): x\in X\}$. So, there exists $x_k\in X$ such that $(\bigvee\{\beta_x: x\in X\})(k)=\beta_{x_k}(k)$. It follows that $\alpha\leq \bigvee\{\beta_{x_k}: x\in F\}$ and $\alpha$ is compact. 
\end{proof}
We obtain the following straight from Proposition \ref{compact-chmv}.
\begin{cor}\label{compact-p} For every nonempty set $X$ and $\{n_x: x\in X\}$ a set of integers greater than or equal to $2$,
$\mathfrak{k}([0,1]^X)=\{0\}$ and $\mathfrak{k}(\prod_{x\in X}\L_{n_x})=\oplus_{x\in X}\L_{n_x}$
\end{cor}
\begin{cor}
Every MV-frame of the form in Proposition \ref{compact-chmv} has the FIP.
\end{cor}
\begin{definition}
Let $L$ be an algebraic frame. An element $a\in \mathfrak{k}(L)$ is called maximal compact if it is a maximal element in $(\mathfrak{k}(L), \leq)$.
\end{definition}
We can now prove our first result characterizing all algebraic MV-frames.
\begin{thm}\label{algebraicMV}
Algebraic MV-frames are up to isomorphism the direct products of finite MV-chains.
\end{thm} 
\begin{proof}
$\Rightarrow):$ Suppose that $A$ is an algebraic frame. Then by Corollary \ref{alg=ch}, $A=\prod_{k\in K}A_k$, where each $A_k$ is either a finite MV-chain or the standard MV-algebra $[0,1]$. Let $I=\{k\in K: A_k \; \mbox{is a finite}\; MV-\mbox{chain}\}$ and $J=\{k\in K:A_k=[0,1]$. We need to prove that $J=\emptyset$ . By contradiction, suppose that there exists $k_0\in K$ such that $A_{k_0}=[0,1]$. Define $\alpha\in A$ by: 
$$
\alpha(k)=
\begin{cases}
1, \mbox{if}\; k=k_0\\
0,\; \mbox{if}\; k\ne k_0
\end{cases}
$$
Note that by Proposition \ref{compact-chmv}, $0$ is the only compact element of $A$ that is below $\alpha$. It follows that $\alpha$ cannot be the supremum of the set of compact below it. Therefore, $A$ is not algebraic.\\
$\Leftarrow):$ Conversely, assume that $A:=\prod_{x\in X}\L_{n_x}$. For each $\alpha\in A$ and $t\in X$, define $\beta_t\in A$:
$$
\beta_t(x)=
\begin{cases}
\alpha(t), \mbox{if}\; x=t\\
0,\; \mbox{if}\; x\ne t
\end{cases}
$$
Then $\beta_t$ is compact for every $t\in X$ again by Proposition \ref{compact-chmv} and $\beta_t\leq \alpha$. Clearly $\alpha=\bigvee\{\beta_t: t\in X\}$, from which it follows that $\alpha=\bigvee\{\beta\in \mathfrak{k}(A): \beta\leq \alpha\}$. Thus, $A$ is algebraic as needed. 
\end{proof}
By the preceding Theorem, we are discovering that whether one considers algebraic MV-frames, profinite MV-algebras \cite[Thm. 2.5]{jbn}, or Stone MV-algebras \cite[Thm. 2.3]{jbn2}, one is dealing with the exact same class of MV-algebras.

Important characterizations of regular algebraic frames can be found in \cite[Theorem 2.4(a)]{MZ1} and they state in part that an algebraic frame $L$ is regular if and only if $a\vee a^\ast=1$, for all $a\mathfrak{k}(L)$.
\begin{cor}\label{regular}
Let $A$ be an algebraic MV-frame. Then, $A$ is regular if and only if $A$ is isomorphic to a powerset (Boolean) algebra.
\end{cor}
\begin{proof}
Suppose that $A$ is an algebraic MV-frame that is regular. Then, by Theorem \ref{algebraicMV}, we may write $A=\prod_{x\in X}\L_{n_x}$ and by \cite[Theorem 2.4(a)]{MZ1}, $a\vee a^\ast=1$ for all $a\in \mathfrak{k}(A)$. Suppose that there exists $x_0\in X$ and $0<t<1$ in $\L_{n_{x_0}}$. Consider $\alpha\in A$ defined by $\alpha(x_0)=t$ and $\alpha(x)=0$ for all $x\ne x_0$. Then $\alpha\in \mathfrak{k}(A)$ by Proposition \ref{compact-chmv} and $(\alpha^\ast)(x_0)=0$ and $\alpha^\ast(x)=1$ for all $x\ne x_0$. It follows that $(\alpha\vee \alpha^\ast)(x_0)=t$ and $\alpha\vee \alpha^\ast<1$. This contradicts the fact that $A$ is regular. Therefore, for every $x\in X$, $n_x=2$ and $\L_{n_x}=\mathbf{2}$, the two-element Boolean algebra. Thus, $A=\prod_{x\in X}\mathbf{2}=\mathbf{2}^X\cong \mathcal{P}(X)$. Conversely, it is that every powerset algebra  satisfies $a\vee a^\ast=1$ for all $a\in \mathfrak{k}(A)$ and is regular by  \cite[Theorem 2.4(a)]{MZ1}.

\end{proof}
\begin{cor}\label{coherent}
The only coherent MV-frames are finite MV-algebras.
\end{cor}
\begin{proof}
Since finite MV-algebras are finite direct products of finite MV-chains \cite[Prop. 3.6.5]{CDM}, it is clear that these are coherent. Conversely, suppose that $A$ is a coherent MV-frame. As $A$ is algebraic, then by Theorem \ref{algebraicMV}, $A\cong \prod_{x\in X}\L_{n_x}$, for some set $X$ and a set of integers $(n_x)_{x\in X}$. Since $A$ is compact, $\prod_{x\in X}\L_{n_x}$ is compact. Note that if $X$ is infinite, then $1:=(1)_{x\in X}\notin  \oplus_{x\in X}\L_{n_x}$, which means by Corollary \ref{compact-p} that $1$ is not compact. This is contradictory to $\prod_{x\in X}\L_{n_x}$ is compact. Therefore, $X$ is finite and it follows that $A$ is finite.
\end{proof}
As observed above, $\{\L_{n_x}:x\in X\}$  $\prod_{x\in X}\L_{n_x}$ is a Stone MV-algebra under the product of the discrete topologies. We seek to characterize the coherent maps between algebraic MV-frames.\\
We start with the following Remark.
\begin{rem}\label{maxk}
Let $A:=\prod_{x\in X}\L_{n_x}$ be an algebraic MV-frame. The maximal compact elements of $A$ are of the form $\chi_F$, (the characteristic function of $F$ in $X$), for some finite subset $F$ of $X$. Indeed as $\mathfrak{k}(A)=\oplus_{x\in X}\L_{n_x}$, then the maximal elements of $\prod_{x\in X}\L_{n_x}$ are of the stated form.
\end{rem}
\begin{prop}\label{coherentmap}
Let $A, B$ be algebraic MV-frames and $\varphi:A\to B$ be an MV-homomorphism. Then the following assertions are equivalent:
\begin{itemize}
\item[(i).] $\varphi$ is coherent;
\item[(ii).] $\varphi$ is complete and preserves maximal compact elements;
\item[(ii).] $\varphi$ is continuous and preserves maximal compact elements.
\end{itemize}
\end{prop}
\begin{proof}
$(i)\Rightarrow (ii):$ Assume that $\varphi$ is a coherent MV-homomorphism. Then $\varphi$ preserves arbitrary suprema. In addition, since $\bigwedge S=\neg (\bigvee \neg S)$ \cite[Lem. 6.6.3]{CDM} and $\varphi$ preserves $\neg$, then $\varphi$ preserves arbitrary infima as well. Hence $\varphi$ is complete. That $\varphi$ and preserves maximal compact elements is clear as it preserves all compact elements\\
$(ii)\Leftrightarrow (iii):$ In light of Theorem \ref{algebraicMV}, this equivalence is part of \cite[Prop. 3.5]{jbn}.\\
$(ii)\Rightarrow (i):$ Assume that $\varphi$ is complete and and preserves maximal compact elements. We only need to prove that $\varphi$ maps compact elements of $A$ to compact elements in $B$. By Theorem \ref{algebraicMV}, we may set $A:=\prod_{x\in X}\L_{n_x}$ and $B:=\prod_{y\in y}\L_{m_y}$. Let $\alpha \in \mathfrak{k}(A)$, then by Corollary \ref{compact-p} $\alpha\in \oplus_{x\in X}\L_{n_x}$. So, $\{x\in X: \alpha(x)\ne 0\}:=F$ is finite. Therefore, there exists an integer $N\geq 2$ such that $N\alpha=\chi_F$, the characteristic function of $F$ in $X$. Note that $\chi_F$ is a maximal compact element of $A$ (Remark \ref{maxk}), so $\varphi(\chi_F)$ is a maximal compact element of $B$. Thus, there exists $G\subseteq Y$ finite such that $\varphi(\chi_F)=\chi_G$. But as $N\alpha=\chi_F$, $N\varphi(\alpha)=\varphi(\chi_F)=\chi_G$ and $N\varphi(\alpha)=\chi_G$. From the latter, it follows that $\varphi(\alpha)\in \oplus_{y\in y}\L_{m_y}$ and $\varphi(\alpha)\in \mathfrak{k}(B)$.
\end{proof}
\begin{ex}
Consider the three-element \L ukasiewicz chain and natural inclusion $\tau: \L_3\to \L_3^{\mathbb{N}}$, that is $\tau(x)(n)=x$ for all $x\in \L_3$ and $n\in \mathbb{N}$. Then, $\tau$ is complete but does not preserve maximal compact elements since $\tau(\frac{1}{2})$ is not even compact. 
\end{ex} 
\begin{ex}
Consider $\varphi: \prod_{n=2}^\infty\L_n\to \prod_{n=2}^\infty\L_{2n-1}$ defined by $\varphi(\alpha)(n)=\alpha(n)$. Note that for every $n\geq 2$, the projection $\prod_{n=2}^\infty\L_n\to \L_n$ is equal to the composition of $\varphi$ followed by the projection $\prod_{n=2}^\infty\L_{2n-1}\to \L_{2n-1}$. This means that $\varphi$ is complete by \cite[Prop. 3.5]{jbn}. In addition, it is clear that $\varphi( \oplus_{n=2}^\infty\L_n)\subseteq \oplus_{n=2}^\infty\L_{2n-1}$, which means that $\varphi$ preserves compact elements. Thus, $\varphi$ is a coherent homomorphism. 
\end{ex}
\section{Nuclei of MV-frames}
A closure operation on an MV-algebra $A$ has the usual meaning, i.e maps from $A\to A$ that are extensive, monotonic nondecreasing, and idempotent.
\begin{definition} Suppose that $j:A\to A$ is a closure operator.
\begin{itemize}
\item[d1.] The sets of fixed points of $j$ is denoted by $jA$. 
\item[d2.] $j$ is called nucleus if $j(a\wedge b)=j(a)\wedge j(b)$, and in this case $jA$ is called nuclear. 
\item[d3.] A nucleus $j$ is dense if $j(0)=0$, that is $0\in jA$.
\item[d4.] If $A$ is an algebraic MV-frame, then $j$ is called inductive if $j(x)=\bigvee\{j(a): a\in \mathfrak{k}(A)\;\mbox{and}\; a\leq x\}$
\end{itemize}
\end{definition}
Note that as $j^2=j$, then $jA=j(A)$.
\begin{ex} 
\begin{enumerate}
\item The identity $j(x)=x$ and the constant map $j(x)=1$ are obvious examples of closure operators on any MV-algebra $A$.
\item The double-pseudocomplementation $j(x)=x^{\ast\ast}$ is a closure operator on any MV-algebra. This is clear from the properties of pseudocomplementation listed in Section 2. In addition, by the same listed properties, $j$ is nucleus and $jA=B(A)$, the Boolean center of $A$.
\item Let $X$ be a topological space and $A:=\mbox{Cont}(X)$ be the MV-algebra of continuous functions from $X\to [0,1]$. Fix $t_0\in [0,1]$ and define $j:A\to A$ by $j(f)(x)=\mbox{Max}(f(x),t_0)$. Then $j$ is a closure operator on $A$. Moreover, $j$ is nucleus and $jA=\{f\in A: f(x)\geq t_0\; \mbox{for all}\; x\in X\}$. 
\item Let $A:= \prod_{n=2}^\infty\L_n$ and define $j:A\to A$ by:
$$
j(\alpha)(n)=
\begin{cases}
\alpha(n), \mbox{if}\; n\; \mbox{is even}\\
1,\; \; \; \; \mbox{if}\; \;  n\; \mbox{is odd}
\end{cases}
$$
Then, $j$ is a nucleus closure operator on $A$. In addition, $jA=\{\alpha\in A: \alpha(2k+1)=1,\; \mbox{for all}\; k\geq 1\}$. Moreover, one can verify that $j$ is inductive.
\end{enumerate}
\end{ex}
As witnessed in the examples above, the nuclear $jA$ may fail to be a sub-MV-algebra of $A$. One would like to consider closure operators for which $jA$ is an algebraic MV-frame. 
\begin{definition}
A nucleus $j:A\to A$ will be called of MV-type if $jA$ is closed under $\neg$ and $\oplus$.
\end{definition}
\begin{ex}
Consider the standard MV-algebra $[0,1]$ (which is clearly an MV-frame) and $j:[0,1]\to[0,1]$ defined by $j(x)=\lceil x \rceil$. Then $j$ is a dense nucleus of MV-type. Note that $j$ is not inductive, however. This Example shows that being of MV-type is stricter weaker than requiring that $j$ preserves the MV-operations. Indeed, $j$ is of MV-type as stated ($j[0,1]=\L_2$) but $j$ does not preserve $\neg$.
\end{ex}
Note that a nucleus $j: A\to A$ is of MV-type if and only if $jA$ is a sub-MV-algebra of $A$. This is because $j(1)=1$ (as $j$ is extensive) and $0=\neg 1$ and since $jA$ is closed under $\neg$, then $j(0)=0$. 
\begin{ex}\label{nucl}
Consider $A:=\prod_{n=1}\L_{2n+1}$. For each $n\geq 1$, consider $j_n:\L_{2n+1}\to \L_{2n+1}$ and $j:A\to A$ defined by:
\begin{itemize}
\item[1.] $j_n(0)=0$,
\item[2.] $j_n(\frac{2k-1}{2n})=j_n(\frac{k}{n})=\frac{k}{n}$ for all $k=1, 2, \ldots, n$,
\item[3.] $j(\alpha)(n)=j_n(\alpha(n))$ for all $\alpha\in A$ and $n=1,2,\ldots $.
\end{itemize}
Then $j$ is a nucleus that is inductive and of MV-type. Indeed, one can verify that $jA=\prod_{n=1}\L_{n+1}$, which is clearly a sub-MV-algebra of $A$.
\end{ex} 
One area of algebra where both frames and nuclei naturally arise is the ideal theory of commutative rings. Indeed, such the ideals of a commutative unitary ring form a frame and closure operators of various types have been studied on this frame. We would like to consider broadly speaking the intersection of these topics and MV-frames. Belluce and Di Nola \cite{BD}  investigated and completely characterized the commutative rings $R$ generated by idempotents for which the frame Id$(R)$ of ideals of $R$ is an MV-algebra, that is an MV-frame. Recall \cite{BD} that in the MV-frame Id$(R)$, one has: $\neg I=I^\ast$, the annihilator of $I$, $I\oplus J=(I^\ast J^\ast)^\ast$, $0:=\{0\}$ and $1:=R$; if $\{I_x\}_{x\in X}\subseteq \mbox{Id}(R)$, then $\bigwedge_{x\in X}I_x=\cap_{x\in X}I_x$ and $\bigvee_{x\in X}I_x=\langle \cup_{x\in X}I_x\rangle$. These rings are called \L ukasiewicz rings and were shown \cite[Thm. 7.7]{BD} to be up to isomorphism the direct sums commutative unitary Artinian chain rings. 

For every \L ukasiewicz ring $R$, the MV-frame $\mbox{Id}(R)$ is algebraic. This follows from the fact that $\mbox{Id}(R)$ is complete and atomic (\cite[Prop. 3.15]{BD}, that complete and atomic MV-algebras are direct products of finite chains (\cite[Cor. 6.8.3]{CDM}) and Theorem \ref{algebraicMV} above.

We can now find some examples of nuclei of on the MV-frame Id$(R)$, where $R$ is a \L ukasiewicz ring either using classical constructions in ring theory or by taking advantage of the fact that Id$(R)$ is a direct product of finite chains. 

Let $R$ be a \L ukasiewicz ring and $A:=\mbox{Id}(R)$, the MV-algebra of ideals of $R$. From the preceding comment, we may write $A=\prod_{x\in X}\L_{n_x}$, for some nonempty set $X$ and $\{n_x:x\in X\}$ a set of integers greater than or equal to $2$. The next result uses these notations.
\begin{prop}\label{radical}
For $I\in \mbox{Id}(R)$, let $\sqrt{I}:=\bigcap\{P\in \mbox{Spec}(R): I\subseteq P\}$.\\
%Then, since prime ideals of $R$ are maximal \cite[Prop. 3.13]{BD}, then $\sqrt{I}:=\bigcap\{M\in \mbox{Max}(R): I\subseteq M\}$. \\
Then, $I\mapsto \sqrt{I}$ is (i) a closure operator, (ii) a nucleus, (iii) is inductive, (iv) is not dense, unless $n_x=2$ for all $x\in X$.
\end{prop}
\begin{proof}
(i) This is known (see, e.g., \cite[Ex. 2.1.2]{EP}(3)).\\
(ii) Given $I, J\in \mbox{Id}(R)$, and $P$ a prime ideal of $R$, $I\cap J\subseteq P$ if and only if $I\subseteq P$ or $J\subseteq P$. It follows from this that every prime ideal containing $I\cap J$ must contain $\sqrt{I}\cap\sqrt{J}$. Thus, $\sqrt{I}\cap\sqrt{J}\subseteq\sqrt{I\cap J}$. The reverse inclusion is clear. \\
(iii) To see that $\sqrt{\; \; }$ is inductive, it is easier to translate this operator explicitly on the MV-algebra $A$ as the MV-algebra of functions $f$ from $X\to \sqcup_{x\in X} \L_{n_x}$. For $f\in A$, $$\sqrt{f}=\bigvee\{\alpha\in A: \alpha \; \mbox{co-atom}\; \mbox{and}\; f\leq \alpha\}$$
The above formula for $\sqrt{f}$ is derived from the fact that prime ideals of $R$ are maximal \cite[Prop. 3.13]{BD} and that the maximal ideals of $R$ correspond to the coatoms in $\prod_{x\in X}\L_{n_x}$.
Now, the co-atoms of $A$ are precisely the functions $\alpha$ for which there exists $x_0\in X$ such that $\alpha(x_0)=\frac{n_{x_0}-2}{n_{x_0}-1}$ and $\alpha(x)=1$ for all $x\neq x_0$. It follows that:
$$
\left(\sqrt{f}\right)(x)=
\begin{cases}
\frac{n_{x}-2}{n_{x}-1},\; \; \; \;  \mbox{if}\; f(x)<1\\
1,\; \; \; \; \; \; \; \; \; \mbox{if}\; \;  f(x)=1
\end{cases}
$$
We wish to show that $$\sqrt{f}=\bigvee\left\{\sqrt{\alpha}:\; \; \; \alpha\in \mathfrak{k}(A)\; \wedge \; \alpha\leq f\right\}$$
We consider two cases:\\
\underline{Case 1:} $f(x)<1$ for all $x$, then $\left(\sqrt{f}\right)(x)=\frac{n_{x}-2}{n_{x}-1}$ for all $x\in X$ . In this case, $\left(\sqrt{f}\right)(x)=\frac{n_{x}-2}{n_{x}-1}$ for all $x\in X$. For each $t\in X$, define $\alpha_t\in A$ by $\alpha_t(t)=f(t)$ and $\alpha_t(x)=0$ for all $x\neq t$. Then, $\alpha_t\in \mathfrak{k}(A)$ (Corollary \ref{compact-p}) and $\alpha_t\leq f$ for all $t\in X$. In addition $\left(\sqrt{\alpha_t}\right)(x)=\frac{n_{x}-2}{n_{x}-1}$ for all $x\in X$. It follows that $\bigvee_{t\in X}\sqrt{\alpha_t}=\sqrt{f}$ and $\bigvee\{\sqrt{\alpha}:\alpha\in \mathfrak{k}(A)\; \wedge \; \alpha\leq f\}=\sqrt{f}$. \\
\underline{Case 2:} The set $S:=\{x\in X:f(x)=1\}\ne \emptyset$. For each $s\in S$, define $\alpha_s\in A$ by $\alpha_s(s)=1$ and $\alpha_s(x)=0$ for all $x\ne s$. Then, $\alpha_t\in \mathfrak{k}(A)$ (Corollary \ref{compact-p}) and $\alpha_t\leq f$ for all $t\in X$. Moreover, $\left(\sqrt{\alpha_s}\right)(s)=1$ and $\left(\sqrt{\alpha_s}\right)(x)=\frac{n_{x}-2}{n_{x}-1}$ if $x\ne s$. It follows that $$\left(\bigvee_{s\in S}\sqrt{\alpha_s}\right)(x)=\begin{cases}
\frac{n_{x}-2}{n_{x}-1},\; \; \;  \mbox{if}\; x\notin S\\
1,\; \; \; \; \; \; \hspace{0.3cm} \mbox{if}\; \;  x\in S
\end{cases}
=\left(\sqrt{f}\right)(x)
$$
Therefore, $$\sqrt{f}=\bigvee\left\{\sqrt{\alpha}:\; \; \; \alpha\in \mathfrak{k}(A)\; \wedge \; \alpha\leq f\right\}$$.

(iv) We use the set-up from (iii). It is clear that if $n_x=2$ for all $x\in X$, then $\sqrt{0}=0$. In addition, if $n_{x_0}>2$ for some $x_0\in X$, then $\left(\sqrt{0}\right)(x_0)=\frac{n_{x_0}-2}{n_{x_0}-1}>0$. Hence, $\sqrt{0}>0$ and $\sqrt{\; \;}$ is not dense. 
\end{proof}
We would like to point out that in the preceding Proposition, the nuclear $\sqrt{A}=\prod_{x\in X}\{\frac{n_{x}-2}{n_{x}-1},1\}$. In particular, while $\sqrt{A}$ is a not a sub-MV-algebra of $A$, unless $n_x=2$ for all $x\in X$.

When $j$ is inductive and of MV-type, we have the following that provides a method of constructing algebraic MV-frames from such nuclei.
\begin{prop}\label{nuclear-p}
If $A$ is any MV-frame and $j:A\to A$ is a nucleus that is inductive of MV-type, then $jA$ is an algebraic MV-frame.
\end{prop}
\begin{proof}
First, we observed before Example \ref{nucl} that $jA$ is a sub-MV-algebra of $A$, therefore an MV-algebra. The completeness of $jA$ is again obtained from \cite[Prop. 7.2]{PD} but in this case the meets in $jA$ and $A$ coincide. Indeed, using the relation between joins and meets in any MV-algebra \cite[Lem. 6.6.3]{CDM}(Eq. 6.8), we obtain for every nonempty subset $S\subseteq jA$, 
 \begin{align*}
\vee_jS&= \neg_j\wedge_j\neg_jS \; \; \; \; \; \; \; \mbox{by} \; (\mbox{Eq.}\; 6.8)\; \mbox{in the MV-algebra}\; jA\\
  &=\neg\wedge\neg S\; \; \; \; \; \; \; \; \mbox{as} \; jA\; \mbox{is a sub-MV-algebra and the meets in}\; jA\; \mbox{and}\; A \; \mbox{coincide}\\
  &=\vee S \; \; \; \; \; \; \; \hspace{0.5cm} \mbox{by} \; (\mbox{Eq.}\; 6.8) \; \mbox{in the MV-algebra}\; A
 \end{align*}
Therefore, $jA$ is an MV-frame, indeed a sub-MV-frame of $A$. 
 
 To see that $jA$ is algebraic, let $x\in jA$. Then,
 \begin{align*}
x&= \bigvee\left\{j(a):a\in \mathfrak{k}(A), a\leq x\right\}\; \; \; \; \; \; \; \; \; ( j\; \mbox{is inductive})\\
  &\leq \bigvee\left\{j(a):a\in \mathfrak{k}(A), j(a)\leq x\right\} \; \; \; \; \; \; \; \; \; ( j(x)=x)\\
  &\leq \bigvee\left\{b:b\in j\mathfrak{k}(A), b\leq x\right\}\\
    &\leq \bigvee\left\{b:b\in \mathfrak{k}(jA), b\leq x\right\} \; \; \; \; \; \; \; \; \; \; \; \;(j\mathfrak{k}(A)=\mathfrak{k}(jA))\\
  &\leq x
 \end{align*}
Hence, $x=\bigvee\left\{b:b\in \mathfrak{k}(jA), b\leq x\right\}$ and $jA$ is algebraic.
\end{proof}
\begin{rem}
Note that one particularity of MV-frames compare to general frames resides in the MV-frame $jA$, the joins in $jA$ coincide with those in $A$.
\end{rem}
\section{$\ell u$-Frames}
In this section, we use the Chang-Mundici equivalence between MV-algebras and abelian $\ell$-groups with strong units to consider $\ell$-groups as frames and some of their properties. This investigation is motivated in part by the fact while there exists an equivalence of categories, the frame notions such as compactness, algebraic many others are not directly categorical concepts. 

Lattice-ordered groups (or $\ell$-groups for short) are groups equipped with a lattice structure that is compatible with the group operations. 
 The only $\ell$-groups that we shall deal with are Abelian, therefore we will use the additive notation $\langle G,+,-,0\rangle$.
 We shall also use the following traditional notations: given $a\in G$, $a^+=a\vee 0$, $a^-=-a\vee 0$ and $|a|=a^-+ a^+=a\vee -a$, in particular $a^+,a^-, |a|\in G^+$.\\
 Given an Abelian $\ell$-group $G$, an element $u\in G^{+}$ is called a strong unit if for all $x\in G$, there exists an integer $n\geq 1$, such that $|x|\leq nu$.\\
 Given an abelian $\ell$-group $\langle G,+,-,0\rangle$ together with a strong unit $u\in G$ (called $\ell u$-group for short), let $\Gamma G:=[0,u]:=\{x\in G:0\leq x\leq u\}$. Then it is known that $\langle \Gamma G, \oplus, \neg, 0\rangle$ is an MV-algebra where for every $x, y\in $, $$x\oplus y=(x+y)\wedge u\; \; \mbox{and}\; \; \neg x=u-x$$
 Indeed, it is established in \cite{CM} that $\Gamma$ defines an equivalence of categories from the category $\mathbb{ABG}_u$ of $\ell u$-groups onto the category of $\mathbb{MV}$ of MV-algebras and its inverse is denoted by $\Phi$.
 
Consider a complete $\ell$-group $G$ (i.e., every nonempty bounded subset of $G$ has a l.u.b and a g.l.b) and $u\in G$ is a strong unit. Then, $G$ satisfies the infinite distributivity law \cite[F.3]{LJMW}, that is: $\wedge$ distributes over all suprema that exists. For this reason, we shall refer to any pair $\langle G, u\rangle$, where $G$ is a complete $\ell$-group and $u\in G$ is a strong as an $\ell u$-frame. In addition, if the frame $G$ has a property $(P)$ (for e.g., algebraic, regular,...etc), we will say that the $\ell u$-frame has property $(P)$. Given $\langle G_1, u_1\rangle$ and $\langle G_2, u_2\rangle$, two $\ell u$-frames, a homomorphism from $\langle G_1, u_1\rangle \to \langle G_2, u_2\rangle$ is both a group and lattice homomorphism from $G_1to G_2$ that preserves arbitrary suprema that exists and maps $u_1$ to $u_2$. It is clear that MV-frames and frame homomorphisms form a category that shall be denoted by $\mathbf{MV}$-$\mathbf{Frm}$ and the $\ell u$-frames and their homomorphisms form a category that shall be denoted by $\ell u$-$\mathbf{Frm}$.  We seek to use the categorical equivalence described above to characterize some classes of $\ell u$-frames. 

Given a lattice-ordered group $G$ and $x\in G$, define $M(x):=\{a\in G: x\leq a\}$ and $L(x):=\{a\in G: a\leq x\}$. The \text{interval topology} is the topology generated by taking all of the
sets $\{M(x), L(x): x\in G\}$, as a subbasis for the closed sets.
\begin{prop}\label{mv=lu}
\begin{itemize}
\item[1.] For every $\ell u$-frame, $\Gamma G$ is an MV-frame.
\item[2.] For every MV-frame $A$, $\Phi(A)$ is an $\ell u$-frame.
\end{itemize}
\end{prop}
\begin{proof}
1. Assume that $\langle G, u\rangle$ is an $\ell u$-frame. Let $(g_x)_{x\in X}\subseteq \Gamma G$, then $(g_x)_{x\in X}$ is bounded and as $G$ is complete, then $\bigvee_{x\in X}g_x$ and $\bigwedge_{x\in X}g_x$ exist in $G$. Clearly $\bigvee_{x\in X}g_x, \bigwedge_{x\in X} \in \Gamma G$ and $\bigvee_{x\in X}g_x$ (resp. $\bigwedge_{x\in X}$) is the supremum (resp. the infimum) of $(g_x)_{x\in X}$ in $\Gamma G$. \\
2. Let $A$ be an MV-frame. Then, $A=\Gamma G$, for some $\ell u$-group $\langle G, u\rangle$. We need to prove that $G$ is complete under the assumption that $[0,u]$ is complete. 
%We shall achieve this using \cite[Thm. 3]{CH} by proving that $[0,a]$ is compact in the interval topology for all $a\in G^{+}$. Consider the translation $f:G\to G$ defined $f(x)=x+u$, which is clearly continuous. Recall that a lattice is complete if and only if it is compact in the interval topology. Now, let $a\in G^{+}$, then as $u$ is a strong unit, there exists an integer $n\geq 1$ such that $a\leq nu$, that is $[0,a]\subseteq [0,nu]$. Note that $f([0,u])=[u,2u]$, $f([u, 2u])=[2u,3u]$,\ldots. It follows as $[0,u]$ is compact, then $[ku,(k+1)u]$ is compact for all $k=0,\dots, n-1$ as the direct image of a compact set under a continuous function is compact. Now, $[0,nu]=\cup_{k=0}^{n-1}[ku,(k+1)u]$, then $[0,nu]$ is compact as a finite union of compact sets. Therefore, since $[0,a]$ is closed in the interval topology, $[0,a]$ becomes a closed subset of the compact set $[0,nu]$. Hence, $[0,a]$ is compact as needed.
Using \cite[Lem. 2]{CH}, we show that every bounded set $(g_x)_{x\in X}\subseteq G^{+}$ has a g.l.b. Let $g\in G$ such that $g_x\leq g$ for all $x\in X$. There exists $n\geq 1$ integer such that $g\leq nu$. Thus, there exists an integer $n\geq 1$ such that $g_x\leq nu$ for all $x\in X$. Choose the smallest integer $N\geq 1$ with the preceding property. For each $x\in X$, as $[0,u]$ is complete, let $b_x:=\bigvee\{a\in[0,u]:Na\leq g_x\}$ and let $b=\bigvee_{x\in X}b_x$. It can be shown using \cite[Lem. 2]{Wright} that $Nb$ is the l.u.b of $(g_x)_{x\in X}$. Therefore, $G$ is complete by \cite[Lem. 2]{CH}.
\end{proof}
\begin{prop}\label{morphisms}
\begin{itemize}
\item[1.] Given $\langle G_1, u_1\rangle$ and $\langle G_2, u_2\rangle$, two $\ell u$-frames and $f: \langle G_1, u_1\rangle \to \langle G_2, u_2\rangle$ a morphism, $\Gamma(f):\Gamma G_1\to \Gamma G_2$ is frame homomorphism.
\item[2.] Given $A, B$ two MV-frames and a frame homomorphism $\phi:A\to B$, $\Phi(\phi): \Phi(A)\to \Phi(B)$ is a homomorphism of $\ell u$-frames.
\end{itemize}
\end{prop}
\begin{proof}
1. Let $f: \langle G_1, u_1\rangle \to \langle G_2, u_2\rangle$ a morphism. Define $\Gamma(f) :[0,u_1]\to[0,u_2]$ to be the restriction of $f$ to $[0,u_1]$, then $\Gamma(f)$ is a well-defined MV-homomorphism. In addition, since the suprema in $[0,u_i]$ are from $G_i$ ($i=1,2$) and $f$ preserves the suprema, then $\Gamma(f)$ preserves the suprema. Thus, $\Gamma(f)$ is a frame homomorphism. \\
2. Let $A, B$ be two MV-frames and a frame homomorphism $\varphi:A\to B$. It is known by Proposition \ref{mv=lu} that there exists two $\ell u$-frames $\langle G_1, u_1\rangle$, $\langle G_2, u_2\rangle$ and a morphism $\varphi:\langle G_1, u_1\rangle\to \langle G_1, u_1\rangle$ such that $A=[0,u_1], B=[0,u_2]$ and $\phi=\varphi_{|[0,u_1]}$. To show that $\varphi=\Phi(\phi)$ preserves suprema, we will use the notations and the description of the supremum in the proof of Proposition \ref{mv=lu}(2). Let $(g_x)_{x\in X}\subseteq G^{+}$ be a bounded subset of $G$,
\begin{align*}
  \varphi\left(\bigvee_x g_x\right)&=  \varphi\left(N\bigvee_x b_x\right)
  = N\phi\left(\bigvee_x b_x\right)
  = N\bigvee_x \phi(b_x)\\
  &=N\bigvee_x \phi\left(\bigvee\{a\in[0,u]:Na\leq g_x\}\right)&&\\
  &=N\bigvee_x \left(\bigvee\{\phi(a)\in[0,u]:Na\leq g_x\}\right)\\
    &=N\bigvee_x \left(\bigvee\{\varphi(a)\in[0,u]:Na\leq g_x\}\right)\\
  &\leq N\bigvee_x \left(\bigvee\{a'\in[0,u]:Na'\leq \varphi(g_x)\}\right)\\
  &=\bigvee_x\varphi(g_x)
 \end{align*}
 Note that inequality $\bigvee_x\varphi(g_x)\leq  \varphi\left(\bigvee_x g_x\right)$ is always true. Thus, $\bigvee_x\varphi(g_x)= \varphi\left(\bigvee_x g_x\right)$. For the general case, if $(g_x)_x$ is a bounded set in $G$ and $z$ is a lower bound of $(g_x)_x$, then $(g_x-z)_x\subseteq G^{+}$ and is bounded. It follows from the preceding step that:\\ $\varphi\left(\bigvee_xg_x\right)-\varphi(z)=\varphi\left(\bigvee_xg_x-z\right)=\varphi\left(\bigvee_x(g_x-z)\right)=\bigvee_x\varphi(g_x-z)=\bigvee\left(\varphi(g_x)-\varphi(z)\right)=\bigvee\varphi(g_x)-\varphi(z)$. From this, one obtains that $\varphi\left(\bigvee_xg_x\right)=\bigvee\varphi(g_x)$.
\end{proof}
\begin{thm}\label{mvf=luf}
The categories $\mathbf{MV}$-$\mathbf{Frm}$ of MV-frames and $\ell u$-$\mathbf{Frm}$ of $\ell u$-frames are categorically equivalent. 
\end{thm}
\begin{proof}
Recall \cite{CM} that $\Gamma$ defines an equivalence of categories from the category $\mathbb{ABG}_u$ of $\ell u$-groups onto the category of $\mathbb{MV}$ of MV-algebras and its inverse is denoted by $\Phi$. In addition, by Proposition \ref{mv=lu} and Proposition \ref{morphisms}, the restriction of $\Gamma$ to $\ell u$-$\mathbf{Frm}$ is a functor onto $\mathbf{MV}$-$\mathbf{Frm}$ and the restriction of $\Phi$ to $\mathbf{MV}$-$\mathbf{Frm}$ is also a functor onto $\ell u$-$\mathbf{Frm}$. It is therefore clear that $\Gamma$ is an equivalence from $\ell u$-$\mathbf{Frm}\to \mathbf{MV}$-$\mathbf{Frm}$ 
\end{proof}
\begin{definition} Let $\langle G, u\rangle$ be an $\ell u$-frame and $A:=\Gamma(\langle G, u\rangle)$. 
\begin{enumerate}
\item An element $g\in G$ is called compact if $|g|\wedge u$ is a compact element of the MV-frame $\Gamma(\langle G, u\rangle)$. That is,
$$\mathfrak{k}(\langle G, u\rangle)=\{g\in G: |g|\wedge u\in \mathfrak{k}(A)\}$$
\item The $\ell u$-frame $\langle G, u\rangle$ is algebraic if for every $g\in G$, $g=\bigvee\{a\in \mathfrak{k}(\langle G, u\rangle): a\leq g\}$.
\end{enumerate}
\end{definition}
%\begin{prop}Let $\langle G, u\rangle$ be an $\ell u$-frame and $A:=\Gamma(\langle G, u\rangle)$. Then, 
%\begin{enumerate}
%\item $A$ is algebraic if and only if $\langle G, u\rangle$ is algebraic.
%\item If $A$ is algebraic, then $A$ is regular if and only if $\langle G, u\rangle$ is regular.
%\end{enumerate}
%\end{prop}
We obtain the following characterization of algebraic $\ell u$-frames.
 \begin{thm}\label{algebraicluf}
 Algebraic $\ell u$-frames are up to isomorphism of the form $\langle \mathbb{Z}^X, \mathbf{n}\rangle$ for some set $X$, where $\mathbf{n}$ is a sequence of positive integers..
 \end{thm}
 \begin{proof}
 Let $\langle G, u\rangle$ be an algebraic $\ell u$-frame and let $A:=\Gamma(\langle G, u\rangle)$. We prove that $A$ is an algebraic MV-frame. Indeed, let $x\in A$, then $x\in G$ and since $\langle G, u\rangle$ is algebraic, then $x=\bigvee\{a\in \mathfrak{k}(\langle G, u\rangle): a\leq x\}$. Note that for every $a\in \mathfrak{k}(\langle G, u\rangle)$ such that $a\leq x$, $a\leq |a|\wedge u$, since $a\leq |a|$ and $a\leq x\leq u$. Hence, $\bigvee\{a\in \mathfrak{k}(\langle G, u\rangle): a\leq x\}\leq \bigvee\{|a|\wedge u: a\in \mathfrak{k}(\langle G, u\rangle)\; \mbox{and}\; a\leq x\}$. But, as $\{|a|\wedge u: a\in \mathfrak{k}(\langle G, u\rangle)\; \mbox{and}\; a\leq x\}\subseteq \{b\in \mathfrak{k}(A):b\leq x\}$, then $\bigvee\{|a|\wedge u: a\in \mathfrak{k}(\langle G, u\rangle)\; \mbox{and}\; a\leq x\}\leq \bigvee\{b\in \mathfrak{k}(A):b\leq x\}\leq x$. Therefore, $x=\bigvee\{a\in \mathfrak{k}(\langle G, u\rangle): a\leq x\}\leq \bigvee\{|a|\wedge u: a\in \mathfrak{k}(\langle G, u\rangle)\; \mbox{and}\; a\leq x\}\leq \bigvee\{b\in \mathfrak{k}(A):b\leq x\}\leq x$. Thus, $x=\bigvee\{b\in \mathfrak{k}(A):b\leq x\}$ and $A$ is algebraic. It follows from Theorem \ref{algebraicMV} that $A\cong \prod_{x\in X}\L_{n_x}$, for some nonempty set $X$ and a set of positive integers $(n_x)_{x\in X}$. Thus, $\langle G, u\rangle\cong \Phi(A)\cong \Phi\left(\prod_{x\in X}\L_{n_x}\right)\cong \prod_{x\in X}\Phi\left(\L_{n_x}\right)\cong \prod_{x\in X}\langle \mathbb{Z}, n_x-1\rangle\cong \langle \mathbb{Z}^X, \mathbf{n}\rangle$, where $\mathbf{n}=(n_x-1)_{x\in X}$. In addition, one can verify that $\mathfrak{k}\left(\mathbb{Z}^X, \mathbf{n}\rangle\right)=\oplus_{x\in X} \mathbb{Z}$ and use this to show that $ \langle \mathbb{Z}^X, \mathbf{n}\rangle$ is algebraic. 
 \end{proof}
 \section{Conclusion and final remarks}
 In this introductory work, we introduce MV-frames and their nuclei as well as $\ell u$-frames. We also investigated some of the main frame concepts in the framework of MV-algebras and abelian lattice-ordered groups. We completely described algebraic MV-frames ($\ell u$-frames), coherent MV-frames and regular MV-frames among algebraic ones. We also studied nuclei on MV-frames, especially on the MV-frame of ideals of \L ukasiewicz rings. In addition, we studied the nuclear of nuclei and under certain conditions obtained them as algebraic sub-MV-frames. We anticipate that some of our future works will look into deepening some of the areas introduced here. For instance, we would like to find out what new properties of \L ukasiewicz rings can be discovered from the properties of the nuclei on their MV-frames of ideals. 
 
\end{document}